\documentclass[10pt,reqno,a4paper]{amsart}

\let\oldtocsection=\tocsection

\let\oldtocsubsection=\tocsubsection

\let\oldtocsubsubsection=\tocsubsubsection

\renewcommand{\tocsection}[2]{\hspace{0em}\oldtocsection{#1}{#2}}
\renewcommand{\tocsubsection}[2]{\hspace{2em}\oldtocsubsection{#1}{#2}}
\renewcommand{\tocsubsubsection}[2]{\hspace{2em}\oldtocsubsubsection{#1}{#2}}

\usepackage{amsmath,amsthm,amsfonts,amssymb}

\usepackage{marginnote}

\usepackage{fancyhdr}

\usepackage{graphicx}

\usepackage{pdfpages}

\usepackage{tikz}
\usetikzlibrary{matrix}
\usetikzlibrary{cd}

\usepackage{xcolor}

\numberwithin{figure}{section}
\numberwithin{equation}{section}

\newtheorem{thm}{Theorem}[section]

\newtheorem{defn}[thm]{Definition}

\newtheorem{lmm}[thm]{Lemma}

\newtheorem{prp}[thm]{Proposition}

\newtheorem{remark}[thm]{Remark}

\newcommand{\tei}{Teichm\"uller}
\newcommand{\qc}{quasiconformal}

\newcommand{\idt}{of $\id$-type}
\newcommand{\nd}{\mathcal{N}_d}

\newcommand{\id}{\operatorname{id}}
\newcommand{\Id}{\operatorname{Id}}
\newcommand{\SV}{\operatorname{SV}}

\renewcommand{\Re}{\operatorname{Re\,}}

\newcommand{\abs}[1]{\left| #1 \right|}

\newcounter{reminder}

\graphicspath{ {./images/}}

\title{Infinite-dimensional Thurston theory\\and transcendental dynamics IV:\\dependence on parameters\\ and escape on (pre-)periodic rays}
\author{Konstantin Bogdanov}

\begin{document}

\begin{abstract}
	We consider transcendental entire functions that are compositions of a polynomial and the exponential for which all singular values escape on disjoint rays. Based on their classification in \cite{IDTT3} we investigate their dependence on parameters, that is, on the potentials and external addresses of singular values. Using continuity argument we generalize the classification in \cite{IDTT3} to the case when dynamic rays containing singular values are (pre-)periodic.
\end{abstract}

\maketitle
	
\tableofcontents

\addtocontents{toc}{\protect\setcounter{tocdepth}{1}}

This is the forth out of four articles publishing the results of the author's doctoral thesis \cite{MyThesis}. The machinery needed for the Thurston iteration on an infinite-dimensional \tei\ space is developed in \cite{IDTT1}. In \cite{IDTT2,IDTT3} we classify transcendental entire functions that are compositions of a polynomial and the exponential for which their singular values escape on disjoint dynamic rays. \cite{IDTT2} covers the general case of such escape when the distance between points on different singular orbits is bounded below, whereas in \cite{IDTT3} one removes this restriction. In the present article we investigate continuity of such families of functions with respect to potentials and external addresses and use continuity argument to extend our classification to the case of escape on the (pre-)periodic rays.

\section{Introduction}

In transcendental dynamics, as in polynomial dynamics, one is interested
in understanding the \emph{escaping set}. For a transcendental entire function $f$, its \emph{escaping set} is defined as
$$I(f)=\{z\in\mathbb{C}: f^n(z)\to\infty\text{ as }n\to\infty\}.$$
Since the Julia set coincides with the boundary of $I(f)$, both Julia and Fatou sets can be reconstructed from $I(f)$, which is thus one of the most fundamental invariant sets for the dynamics of entire functions. It is shown in \cite{RRRS} for transcendental entire functions of bounded type and finite order that the escaping set of every function in the family is organized as the union \emph{dynamic rays} and their endpoints. About the escaping points that belong to a dynamic ray one says that they \emph{escape on rays}. This is the standard mode of escape for many important families of transcendental entire functions.

In \cite{IDTT2,IDTT3} we developed a classification of entire functions of the form $p\circ\exp$ where $p$ is a monic polynomial such that all their singular points escape on disjoint rays (i.e.\ every dynamic ray contains no more than one post-singular point). Denote
$$\mathcal{N}_d:=\{p\circ\exp: p\text{ is a monic polynomial of degree }d\}.$$ 
As for the exponential family \cite{SZ-Escaping}, for functions in $\nd$ the points escaping on rays can be described by their potential (or ``speed of escape'') and external address (or ``combinatorics of escape'', i.e.\ the sequence of dynamic rays containing the escaping orbit). So the classification was accomplished in terms of potentials and external addresses. More precisely, the classification was obtained separately within each \emph{parameter space} (corresponding to functions in $\nd$). The parameter space of an entire function $g$ is the set of entire functions $\varphi\circ f\circ\psi$ where $\varphi,\psi$ are \qc\ self-homeomorphisms of $\mathbb{C}$. For the exponential family $\mathcal{N}_1$ this definition of the parameter space coincides with the parametrization of $e^z+\kappa$ (up to affine conjugation). Earlier classification results for $\mathcal{N}_1$ were obtained in \cite{FRS,MarkusParaRays,MarkusThesis}.

In this article we prove two results. First, we extend the classification \cite[Theorem~1.1]{IDTT3} of functions in $\nd$ to the case of (pre-)periodic external addresses.

\begin{thm}[Classification Theorem]
	\label{thm:main_thm}
	Let $f_0\in\mathcal{N}_d$ be a transcendental entire function with singular values $\{v_i\}_{i=1}^m$. Let also $\{\underline{s}_i\}_{i=1}^m$ be exponentially bounded external addresses that are non-overlapping, and $\{T_i\}_{i=1}^m$ be real numbers such that $T_i>t_{\underline{s}_i}$. Then in the parameter space of $f_0$ there exists a unique entire function $g=\varphi\circ f_0\circ\psi$ such that each of its singular values $\varphi(v_i)$ escapes on rays, has potential $T_i$ and external address $\underline{s}_i$.
	
	Conversely, every function in the parameter space of $f_0$ such that its singular values escape on rays without orbit overlaps is one of these. 
\end{thm}

The condition that external addresses are non-overlapping means that the dynamic rays corresponding to different singular values of $g$ do not have a common image under $f$, i.e.\ there is no overlap of orbits of these rays (for details see Subsection~\ref{subsec:escaping_set}). 

\begin{remark}
	The restriction to the non-overlapping external addresses is used largely for notational convenience. The machinery described in \cite{IDTT1,IDTT2,IDTT3} could also be applied in full generality though it would lead to even more complicated notation and consideration of quite a few special cases which are not interesting from the conceptual point of view. 
\end{remark}

The second result is about existence of higher dimensional analogues of parameter rays. It is a special case of Theorem~\ref{thm:continuity} which declares a more general continuous dependence on parameters that is presented in Theorem~\ref{thm:parameter rays} (but requires somewhat more preliminaries for stating it precisely).

\begin{thm}[Multidimensional parameter rays]
	\label{thm:parameter rays}
	Let $f_0\in\mathcal{N}_d$ be a transcendental entire function with singular values $\{v_i\}_{i=1}^m$. Let also $\{\underline{s}_i\}_{i=1}^m$ be exponentially bounded external addresses that are non-overlapping. For every $m$-tuple $T=(T_1,...,T_m)$ of real numbers such that $T_i>t_{\underline{s}_i}$ denote by $g_T\in\nd$ the function $\varphi_T\circ f_0\circ\psi_T$ from Theorem~\ref{thm:main_thm}, that is, such that each of its singular values $\varphi_T(v_i)$ escapes on rays, has potential $T_i$ and external address $\underline{s}_i$. Then $g_T$ depends continuously on $T$ (in the sense of coefficients' continuity). 	
\end{thm}

The proof of both theorems is based on the fact that the nearby parameters make possible a ``similar'' construction of the invariant compact subset \cite[Theorem~4.1]{IDTT3} for both parameters, even though they are contained in different \tei\ spaces. This is made precise in Proposition~\ref{prp:invariant_subspiders}.

\section{Prerequisites}

In this section we assemble definitions and results mostly from \cite{IDTT1,IDTT2,IDTT3} that are needed in our constructions.

\subsection{Escaping set of functions in $\mathcal{N}_d$}

\label{subsec:escaping_set}

From \cite{RRRS} we know that for every function $f\in\nd$ its escaping set is modeled on dynamic rays. In particular, that every escaping point is mapped on a dynamic ray after finitely many iterations. A related notion is a ray tail.

\begin{defn}[Ray tails]
	\label{dfn:ray_tail}
	Let $f$ be a transcendental entire function. A \emph{ray tail} of $f$ is a continuous curve $\gamma:[0,\infty)\to I(f)$ such that for every $n\geq 0$ the restriction $f^n|_\gamma$ is injective with $\lim_{t\to\infty}f^n(\gamma(t))=\infty$, and furthermore $f^n(\gamma(t))\to\infty$ uniformly in $t$ as $n\to\infty$.
\end{defn}

\begin{defn}[Dynamic rays, escape on rays]
	\label{dfn:dynamic_ray}
	A \emph{dynamic ray} of a transcendental entire function $f$ is a maximal injective curve $\gamma:(0,\infty)\to I(f)$ such that $\gamma|_{[t,\infty)}$ is a ray tail for every $t>0$.
	
	If $z\in I(f)$ belongs to a dynamic ray, we say that $z$ \emph{escapes on rays}. 
\end{defn}

We showed in \cite{IDTT2} that for functions in $\nd$ its dynamic rays (at least their parts near $\infty$) are characterized by an exponentially bounded external address and parametrized by potential.

Fix an integer $d>0$, and denote by $F:\mathbb{R^+}\to\mathbb{R^+}$ the function $$F(t):=\exp(dt)-1.$$
The \emph{external address} $\underline{s}=(s_0 s_1 s_2 ... )$ is an infinite sequence of integers $s_0,s_1,s_2,...$ On the set of all external addresses we can consider the standard shift-operator $\sigma:(s_0 s_1 s_2 ... )\mapsto (s_1 s_2 s_3 ... )$. Two external addresses $\underline{s}_1$ and $\underline{s}_2$ are called \emph{overlapping} if there are integers $k,l\geq0$ so that $\sigma^k\underline{s}_1=\sigma^l\underline{s}_2$.

Further, an external address $\underline{s}=(s_0 s_1 s_2 ... )$ is called \emph{exponentially bounded} if there exists $t>0$ such that $s_n/{F^n(t)}\to 0$ as $n\to\infty$. The infimum of such $t$ is denoted by $t_{\underline{s}}$.	

Exponentially bounded external addresses correspond to dynamic rays (more precisely, to their part near $\infty$) of functions in $\nd$. Furthermore, from the formula~\ref{eqn:as_formula}, dynamic rays look near $\infty$ almost like straight horizontal lines with the imaginary part that can be recovered from their external address.

\begin{thm}[Escape on rays in $\nd$ \cite{IDTT2}]
	\label{thm:as_formula}
	Let $f\in\mathcal{N}_d$. Then for every exponentially bounded external address there exists a dynamic ray realizing it.
	
	If $\mathcal{R}_{\underline{s}}$ is dynamic ray having an exponentially bounded external address $\underline{s}=(s_0 s_1 s_2 ... )$, and no strict forward iterate of $\mathcal{R}_{\underline{s}}$ contains a singular value of $f$, then $\mathcal{R}_{\underline{s}}$ is the unique dynamic ray having external address $\underline{s}$, and it can be parametrized by $t\in(t_{\underline{s}},\infty)$ so that 
	\begin{equation}
		\label{eqn:as_formula}
		\mathcal{R}_{\underline{s}}(t)=t+\frac{2\pi i s_0}{d} + O(e^{-t/2}),
	\end{equation}
	and $f^n\circ \mathcal{R}_{\underline{s}}=\mathcal{R}_{\sigma^n \underline{s}}\circ F^n$. Asymptotic bounds $O(.)$ for $\mathcal{R}_{\sigma^n\underline{s}}(t)$ are uniform in $n$ on every ray tail contained in $\mathcal{R}_{\underline{s}}$.
	
	If none of the singular values of $f$ escapes, then $I(f)$ is the disjoint union of dynamic rays and their escaping endpoints.
\end{thm}

The parametrization~\ref{eqn:as_formula} allows us to define the potential of a point on a dynamic ray: if $z\in\mathcal{R}_{\underline{s}}$ corresponds to the parameter $t$ of formula~\ref{eqn:as_formula}, we say that $t$ is the \emph{potential} of $z$. From Theorem~\ref{thm:as_formula} follows immediately that the potential is defined in a canonical way, and if the potential of $z$ is equal to $t$, then the potential of $f(z)$ is equal to $F(t)$. Further, different points on the same ray have different potentials.

\subsection{Setup of Thurston iteration}

\label{subsec:setup_and_contraction}

In this subsection we give an overview a few important definitions from \cite{IDTT2} and sketch the setup of the Thurston's iteration.

One starts with the construction of a quasiregular function $f$ modeling escaping behavior of singular values. Let $f_0\in\mathcal{N}_d$, and $v_1,...,v_m$ be the finite singular values of $f_0$. Further, let $\mathcal{O}_1=\{a_{1j}\}_{j=0}^\infty,...,\mathcal{O}_m=\{a_{mj}\}_{j=0}^\infty$ be some orbits of $f_0$ escaping on disjoint rays $\mathcal{R}_{ij}$, and let $R_{ij}$ be the part of the ray $\mathcal{R}_{ij}$ from $a_{ij}$ to $\infty$ (containing also $a_{ij}$).

Now, we introduce the \emph{capture}. Choose some bounded Jordan domain $U\subset\mathbb{C}\setminus\bigcup_{\substack{i=\overline{1,m}\\j=\overline{1,\infty}}} R_{ij}$ containing all singular values of $f_0$ and the first point $a_{i0}$ on each escaping orbit $\mathcal{O}_i$. Define an isotopy $c_u:\mathbb{C}\to\mathbb{C}$ through \qc\ maps, such that $c_0=\id$, $c_u=\id$ on $\mathbb{C}\setminus U$, and for each $i=\overline{1,m}$ we have $c_1(v_i)=a_{i0}$. Denote $c=c_1$. Thus, $c$ is a \qc\ homeomorphism mapping singular values to the first points on the orbits $\mathcal{O}_i$.

Define a function $f:=c\circ f_0$. It is a quasiregular function whose singular orbits coincide with $\{\mathcal{O}_i\}_{i=1}^m$. Its post-singular set of $f$ is defined in a standard way: $$P_f:=\cup_{i=1}^m\mathcal{O}_i.$$ We also call $P_f$ the set of \emph{marked points}.

The Thurston's $\sigma$-map acts on the \tei\ space of $\mathbb{C}\setminus P_f$.

\begin{defn}[\tei\ space of $\mathbb{C}\setminus P_f$]
	\label{defn:tei_space}
	The \emph{\tei\ space} $\mathcal{T}_f$ of the Riemann surface $\mathbb{C}\setminus P_f$ is the set of quasiconformal homeomorphisms of $\mathbb{C}\setminus P_f$ modulo post-composition with an affine map and isotopy relative $P_f$.
\end{defn}

The map
$$\sigma:[\varphi]\in\mathcal{T}_f\mapsto[\tilde{\varphi}]\in\mathcal{T}_f$$
is defined in the standard way via Thurston's diagram.

\begin{center}
	\begin{tikzcd}
		\mathbb{C},P_f \arrow[r, "{\tilde{\varphi}}"] \arrow[d, "f=c\circ f_0"]	& \mathbb{C},\tilde{\varphi}(P_f) \arrow[d, "g"] \\
		\mathbb{C},P_f \arrow[r, "{\varphi}"] & \mathbb{C},\varphi(P_f)
	\end{tikzcd}
\end{center}
\vspace{0.5cm}

More precisely, for every \qc\ map $\varphi:\mathbb{C}\to\mathbb{C}$ there exists another \qc\ map $\tilde{\varphi}:\mathbb{C}\to\mathbb{C}$ such that $g=\varphi\circ f\circ\tilde{\varphi}^{-1}$ is entire function. We define $\sigma[\varphi]:=[\tilde{\varphi}]$. For more details and the proof that this setup is well defined see \cite{IDTT1}. 

\subsection{$\Id$-type maps}
\label{subsec:id_type}

Let $f=c\circ f_0$ be the quasiregular function defined in Section~\ref{subsec:setup_and_contraction} where $f_0=p\circ\exp$. The union of post-singular ray tails $S_0=\cup_{i,j} R_{ij}$ is called the \emph{standard spider} of $f$.

\begin{defn}[$\Id$-type maps \cite{IDTT2}]
	\label{defn:id_type}
	A quasiconformal map $\varphi:\mathbb{C}\to \mathbb{C}$ is \idt\ if there is an isotopy $\varphi_u:\mathbb{C}\to\mathbb{C},\ u\in [0,1]$ such that ${\varphi_1=\varphi},\ \varphi_0=\id$ and $\abs{\varphi_u(z)-z}\to 0$ uniformly in $u$ as $S_0\ni z\to \infty$.
\end{defn}

We also say that $[\varphi]\in\mathcal{T}_f$ is \idt\ if $[\varphi]$ contains an $\id$-type map. Further, $\psi_u, u\in[0,1]$ is an \emph{isotopy \idt\ maps} if $\psi_u$ is an isotopy through maps \idt\ such that $\abs{\psi_u(z)-z}\to 0$ uniformly in $u$ as $S_0\ni z\to \infty$.

The following theorem says that the $\sigma$-map is invariant on the subset of $\id$-type points in $\mathcal{T}_f$.

\begin{thm}[Invariance of $\id$-type points \cite{IDTT2}]
	\label{thm:pullback_of_id_type}
	If $[\varphi]$ is \idt, then $\sigma[\varphi]$ is \idt\ as well.
	
	More precisely, if $\varphi$ is \idt, then there is a unique $\id$-type map $\hat{\varphi}$ such that $\varphi\circ f\circ\hat{\varphi}^{-1}$ is entire.
	
	Moreover, if $\varphi_u$ is an isotopy of $\id$-type maps, then the functions $\varphi_u\circ f\circ\hat{\varphi}_u^{-1}$ have the form $p_u\circ\exp$ where $p_u$ is a monic polynomial with coefficients depending continuously on $u$.
\end{thm}

\begin{remark}
	We keep using the hat-notation from Theorem~\ref{thm:pullback_of_id_type}: $\hat{\varphi}$ denotes the unique $\id$-type map so that $\varphi\circ f\circ\hat{\varphi}^{-1}$ is entire.	
\end{remark}

\subsection{Spiders}

We also make use of infinite-legged spiders which provide a more suitable way of to describe points of $[\mathcal{T}_f]$. The image of a ray tail $R_{ij}$ under a spider map is called a \emph{leg}.

\begin{defn}[Spider]
	\label{defn:spider}
	An image $S_{\varphi}$ of the standard spider $S_0$ under an $\id$-type map $\varphi$ is called a \emph{spider}.	
\end{defn}

For every leg $\varphi(R_{ij})$ of a spider $S_\varphi$ we associate its \emph{leg homotopy word} $W_{ij}^\varphi$ which is an element of certain free group. The leg homotopy word help to measure the wiggling of the spider leg in certain homotopy type, for details see \cite[Definition~4.9]{IDTT2}. As $W_{ij}^\varphi$ is an element of a free group, by $\abs{W_{ij}^\varphi}$ we denote the length of the word representing this element.

\begin{defn}[Equivalence of spiders \cite{IDTT1}]
	We say that two spiders $S_\varphi$ and $S_\psi$ are \emph{equivalent} if for all pairs $i,j$ we have $\varphi(a_{ij})=\psi(a_{ij})$ and legs with the same index are homotopic, i.e.\ $[\varphi(R_{ij})]=[\psi(R_{ij})]$ in $\mathbb{C}\setminus \varphi(P_f)$.
\end{defn}

For any integer $k\geq 0$, we denote by a finite spider $S_\varphi^k$ the image under $\varphi$ of the first $k+1$ legs on each singular orbit, i.e.\ $\varphi(R_{ij}), j\leq k$. One can also generalize equivalence to finite spider: two spiders $S_\varphi^k$ and $S_\psi^k$ are \emph{equivalent} if for all pairs $i,j$ with $j\leq K$, we have $\varphi(a_{ij})=\psi(a_{ij})$ and legs with the same index are homotopic, i.e.\ $[\varphi(R_{ij})]=[\psi(R_{ij})]$ in $\mathbb{C}\setminus \{\varphi(a_{ij}):j\leq K\}$.  

\subsection{Clusters}
\label{subsec:clusters}

We consider the \qc\ function $f=c\circ f_0$ with $m$ singular orbits $\mathcal{O}_i=\{a_{ij}\}$ constructed in Subsection~\ref{subsec:setup_and_contraction} with its associated $\sigma$-map. We denote by $t_{ij}$ the potential of $a_{ij}$ and by $s_{ij}$ the first number in the external address of $a_{ij}$.

Let $\mathcal{P}=\{t_i\}_{i=1}^\infty$ be the set of all potentials of points in $P_f$ ordered so that $t_i<t_{i+1}$. Further, we define a set $$\mathcal{P}':=\{\rho_i:\rho_i=\frac{t_i+t_{i+1}}{2}\}.$$

For all $\rho\in\mathcal{P}'$ we use the following notation. Denote $D_\rho:=\mathbb{D}_\rho (0)$, and for $i\in\{1,...,m\}$ let $N_i=N_i(\rho)$ be the maximal $j$-index of the points $a_{ij}$ contained in $D_\rho$. For $\rho>t'$ the disk $D_\rho$ contains first $N_i+1$ points $\{a_{i0},...,a_{iN_i}\}$ of $\mathcal{O}_i$, whereas the other points of $\mathcal{O}_i$ are in $\mathbb{C}\setminus{D}_{\rho}$.

\begin{defn}[Cluster \cite{IDTT2}]
	\label{defn:cluster}
	We say that $a_{ij}$ and $a_{kl}$ are in the same \emph{cluster $Cl(t,s)$} if they have the same potential $t=t_{ij}=t_{kl}$ and the same first entry $s=s_{ij}=s_{kl}$ in the external address.
	
	A cluster is called non-trivial if it contains more than on point of $P_f$.
\end{defn}

The set $P_f$ is a disjoint union of clusters and each cluster contains at most $m$ points. As can be seen from the asymptotic formula~\ref{eqn:as_formula}, for points $a_{ij}$ and $a_{kl}$ with big potentials being in the same cluster implies that distance between them is very small, whereas the distance between any pair of clusters is bounded from below.

Denote by $\SV(f)$ the set of finite singular values of $f$, i.e.\ the image of (finite) singular values of $f_0$ under the capture $c$. Since we consider only the parameter space of $f_0$, without loss of generality we might assume that $\SV(f_0)\cap I(f_0)=\emptyset$. Then Theorem~\ref{thm:as_formula} claims that all exponentially bounded external addresses and potentials are realized.

We define two types of sets that will help us to describe the behavior of points inside of a cluster under the Thurston's iteration.

\begin{defn}[$A_{x,n}$ \cite{IDTT3}] 
	\label{defn:Axn}
	For $x>0$ and integer $n\geq 0$ let $A_{x,n}$ be the set of all complex numbers $\alpha$ such that
	\begin{enumerate}
		\item $\abs{\log\abs{\alpha}}<e^{-x/3}+...+e^{-F^n(x)/3}$,
		\item $\abs{\arg \alpha}<e^{-x/3}+...+e^{-F^n(x)/3}$.
	\end{enumerate}
\end{defn}

Visually $A_{x,n}$ is an annular sector containing $1$. The next lemma is the direct corollary of Definition~\ref{defn:Axn}.

\begin{lmm}[Pseudo-Multiplicativity of $A_{x,n}$ \cite{IDTT3}]
	\label{lmm:multiplicativity}
	If $\alpha_i\in A_{F^{i-1}(x),0}$ for $i\in\{1,...,k\}$, then $\alpha_1\cdot...\cdot\alpha_k\in A_{x,k-1}$.
\end{lmm}

The next notion helps to characterize the difference between marked point that have the same potential but are in different clusters.
\begin{defn}[$D_{ij}^{kl}$ \cite{IDTT3}]
	\label{defn:Dijkl}
	For $a_{ij}$ and $a_{kl}$ having the same potential  but belonging to different clusters define 
	$$D_{ij}^{kl}:=\left\{\frac{d(w-z)}{2\pi i(s_{kl}-s_{ij})}: w\in\mathbb{D}_{1/l}(a_{kl}), z\in\mathbb{D}_{1/j}(a_{ij})\right\}.$$	
\end{defn}

Next lemma provides an estimate on how the Thurston iteration changes the relative position of marked points inside of the clusters with high potentials.

\begin{prp}[Negligible rotation within clusters \cite{IDTT3}]
	\label{prp:good_big_disk_around_origin_cluster_estimates}
	For all $\rho\in\mathcal{P}'$ big enough holds the following statement.
	
	If $a_{ij},a_{kl}\notin{D}_\rho$ belong to the same cluster $Cl(t,s)$, and $\varphi_u, u\in[0,1]$ is an isotopy \idt\ maps such that $\varphi_0=c^{-1}$ and for all $u\in[0,1]$:
	\begin{enumerate}
		\item $\varphi_u(\SV(f))\subset D_\rho$,
		\item $\abs{\varphi_u(a_{i(j+1)})-a_{i(j+1)}}<1/(j+1)$,
		\item $\abs{\varphi_u(a_{k(l+1)})-a_{k(l+1)}}<1/(l+1)$,
	\end{enumerate}
	then for every $u\in[0,1]$ holds
	$$\hat{\varphi}_u(a_{kl})-\hat{\varphi}_u(a_{ij})=\frac{\varphi_u(a_{k(l+1)})-\varphi_u(a_{i(j+1)})}{F'(t)}\alpha_u,$$
	where $\alpha_u\in A_{t,0}$.
\end{prp}

Some more notation is needed.

\begin{defn}[$H,L$]
	For every pair $a_{ij},a_{kl}$ belonging to the same cluster define
	\begin{itemize}
		\item $H=H(a_{ij},a_{kl})$ to be the smallest positive integer such that $a_{i(j+H)}$ and $a_{k(l+H)}$ belong to different clusters,
		\item $L=L(a_{ij},a_{kl})$ to be the smallest positive integer such that $a_{i(j-L)}$ and $a_{k(l-L)}$ are defined and belong to different clusters, if there is no such integer we define $L:=\infty$.
	\end{itemize}	
\end{defn}

\subsection{Compact invariant subset}
\label{subsec:inv_compact_subset}

Of special importance for us is the construction of the compact invariant subset $\mathcal{C}_f(\rho)$. We agree that $\mathcal{O}_1=\{a_{1j}\}_{j=0}^\infty$ is the orbit of the asymptotic value of $f$.

\begin{thm}[Invariant subset \cite{IDTT3}]
	\label{thm:invariant_subset}
	Let $f=c\circ\exp$ be the quasiregular function defined in Subsection~\ref{subsec:setup_and_contraction}. Further, let $\rho\in\mathcal{P}'$ and $\mathcal{C}_f(\rho)\subset\mathcal{T}_f$ be the \emph{closure} of the set of points in the \tei\ space $\mathcal{T}_f$ represented by $\id$-type maps $\varphi$ for which there exists an isotopy $\varphi_u, u\in[0,1]$ \idt\ maps such that $\varphi_0=\id$, $\varphi_1=\varphi$, and the following conditions are simultaneously satisfied.
	\begin{enumerate}		
		\item (Marked points stay inside of $D_\rho$) If $j\leq N_i$,
		$$\varphi_u(a_{ij})\in D_\rho.$$
		\item (Precise asymptotics outside of $D_\rho$) If $j>N_i$, then
		$$\abs{\varphi_u(a_{ij})-a_{ij}}<1/j.$$		
		\item (Separation inside of $D_\rho$) If $j\leq N_i,l\leq N_k,ij\neq kl$, and $n=\min\{N_i+1-j, N_k+1-l\}$, then $$\abs{\varphi_u(a_{kl})-\varphi_u(a_{ij})}>\frac{\beta\left(a_{i(j+n)},a_{k(l+n)}\right)}{(M_\rho)^n}.$$	
		\item (Bounded homotopy) If $j\leq N_i$, then 		
		$$\abs{W_{ij}^{\varphi_u}}<A^{N_i+1-j}\left(\frac{(N_i+1)!}{j!}\right)^4 C$$		
		where $A$ and $C$ are constants.
		\item (Rigidity in clusters) For all $a_{ij},a_{kl}$ with $j>N_i,l>N_k$ belonging to the same cluster, holds
		$$\varphi_u(a_{kl})-\varphi_u(a_{ij})=\frac{2\pi i(s_{k(l+H)}-s_{i(j+H)})}{d(F^{\circ H}(t))'}\nu_u\delta_u$$
		where $t$ is the potential of $a_{ij}$ and $a_{kl}$, $H=H(a_{ij},a_{kl})$, $\nu_u=\nu_u(\varphi_u,i,j,k,l)\in A_{t,H-1}$ and $\delta_u=\delta_u(\varphi_u,i,j,k,l)\in D_{i(j+H)}^{k(l+H)}$.
		\item (Clusters inside of $D_\rho$) If $a_{k(N_k+1)}$ and $a_{1(N_1+1)}$ belong to the same cluster, $L=L(a_{k(N_k+1)},a_{1(N_1+1)})$, and $N=\max_{i\leq m}N_i$, then:\\ 
		\begin{enumerate}
			\item if $0\leq n<L$, then
			$$\abs{\varphi_u(a_{k(N_k+1-n)})-\varphi_u(a_{1(N_1+1-n)})}<
			4\beta\left(a_{k(N_k+1)},a_{1{(N_1+1)}}\right)M_\rho^{2dNn},$$
			
			\item if $n\geq L$, then
			$$\abs{\varphi_u(a_{k(N_k+1-n)})-\varphi_u(a_{1(N_1+1-n)})}>\left(\frac{1}{M_\rho}\right)^{2d^4N+n-L}.$$
		\end{enumerate}
	\end{enumerate}
	
	Then if $\rho\in\mathcal{P}'$ is big enough, the subset $\mathcal{C}_f(\rho)$ is well-defined, invariant under the $\sigma$-map and contains $[\id]$.	
\end{thm}

We are mostly interested in conditions (1),(2) and (5), so we omit giving precise definitions of $\beta,A,C$ and $M_\rho$ --- they are not needed for the purposes of this article. For details and their definitions we send the reader to \cite{IDTT3}. 

Note that conditions (1)-(2) say that the maps $\varphi$ are ``uniformly \idt'', that is, the marked points outside of a disk $D_\rho$ have precise asymptotics, while inside of $D_\rho$ we allow some more freedom. A the same time, condition (5) helps us to control the behavior of marked points within clusters. It prevents points within clusters to approach and rotate around each other.

\section{Continuity in parameters and (pre-)periodicity}

Let $\mathcal{S}=\mathbb{Z}^{\mathbb{N}}$ be the topological space consisting of all external addresses equipped with the product topology, and $\mathcal{S}_0$ be its subspace consisting of all exponentially bounded external addresses.

\begin{defn}[Wedge $W_{t,K}$]
	For an exponentially bounded external address $\underline{s}$, potential $t>t_{\underline{s}}$ and constant $K>0$, a \emph{wedge} $W_{t,K}$ around $\underline{s}$ in $\mathcal{S}$ is the set of external addresses $\underline{s}'$ such that for all $i>0$ holds
	$$\abs{\underline{s}'_i-\underline{s}_i}^{2d}<KF^i(t).$$
\end{defn} 

Clearly, a wedge around an exponentially bounded external address contains only exponentially bounded external addresses. First, we prove two preliminary lemmas about dependence of dynamic rays on their external addresses.

\begin{lmm}[Continuity of rays w.r.t.\ external address]
	\label{lmm:continuity_rays_external_address}
	Let $f\in\nd$ be such that $\SV(f)\notin I(f)$, and $\tau>t_{\underline{s}}$ be a potential for the external address $\underline{s}\in\mathcal{S}_0$. If $\{\underline{s}^n\}_{n=1}^\infty$ is a sequence of external addresses contained in a wedge $W_{\tau,K}$ around $\underline{s}$ and converging to $\underline{s}$, then
	\begin{enumerate}
		\item $\abs{\mathcal{R}_{\underline{s}^n}(t)-\mathcal{R}_{\underline{s}}(t)}\to 0$ as $n\to\infty$ uniformly on $[\tau,\infty)$,
		\item $\abs{\Re\mathcal{R}_{\sigma^k\underline{s}^n}(t)-t}\to 0$ as $n\to\infty$ uniformly in $t$ and $k\geq 0$.
	\end{enumerate} 
\end{lmm}
\begin{remark}
	Note that $\Re\mathcal{R}_{\sigma^k\underline{s}^n}(t)$ is defined only on $t_{\sigma^k\underline{s}^n}$ so we consider the ray whenever it is defined (depending on $k$).
\end{remark}
\begin{proof}
	For the proof one has to notice that the $O(.)$-estimates for $\delta_k$ in the proof of \cite[Theorem~2.12]{IDTT2} (i.e.\ of Theorem~\ref{thm:as_formula}) do not change if we replace $\underline{s}$ by $\underline{s}'$ in a wedge, i.e.\ the $O(.)$ estimate is uniform in this neighborhood. Hence, for all $i$ bigger than some fixed constant and all $n$, the ray tail $\mathcal{R}_{\sigma^i\underline{s}^n}|_{[\tau,\infty)}$ is fully contained in its tract (see \cite[Lemma~2.2]{IDTT2} for the definition of tracts) and also holds the asymptotic formula~\ref{eqn:as_formula} for each $\underline{s}^n$ with uniform $O(.)$-estimates. Since $f^{-1}$ is uniformly strictly contracting on tracts, the claim follows.
\end{proof}

\begin{lmm}[Continuity of rays w.r.t.\ external address and function]
	\label{lmm:continuity_rays_function}
	Let $h\in\mathcal{N}_{f_0}$ and $\tau>t_{\underline{s}}$ be a potential for the external address $\underline{s}\in\mathcal{S}_0$. If $\{\underline{s}^n\}_{n=1}^\infty$ is a sequence of external addresses contained in a wedge $W_{\tau,K}$ around $\underline{s}$ and converging to $\underline{s}$, and $(h_n)_{n=1}^\infty\subset\mathcal{N}_{f_0}$ converges to $h$, then
	$$\abs{\mathcal{R}_{\sigma^N\underline{s}^n}^{h_n}(t)-\mathcal{R}^{h}_{\sigma^N\underline{s}}(t)}\to 0$$
	as $n\to\infty$ uniformly on $[F^N(\tau),\infty)$ for all $N$ big enough, where $\mathcal{R}_{\sigma^N\underline{s}^n}^{h_n}$ and $\mathcal{R}_{\sigma^N\underline{s}^n}^{h}$ are parametrizations of dynamic rays of functions $h_n$ and $h$, respectively. 
\end{lmm}
\begin{proof}
	The proof is almost the same as of Lemma~\ref{lmm:continuity_rays_external_address}. In \cite[Theorem~2.12]{IDTT2} we constructed the ray parametrizations as limits of functions $g_k(t)$ (see \cite[Theorem~2.12]{IDTT2} for the definition of $g_n$). In particular, one sees that for all $N$ big enough and $k>N$, the difference $h_n^N\circ g_k\circ F^{-n}(t)-\mathcal{R}^{h_n}_{\sigma^N\underline{s}}(t)$ is getting small uniformly in $n$. Claim follows.
\end{proof}

Now, consider a function $f_0\in\nd$ having $m$ singular values. Let $\mathcal{N}_{f_0}$ be the intersection of $\nd$ with the parameter space of $f_0$. Further, for the topological space $$\underbrace{\mathcal{S}\times...\times\mathcal{S}}_{m\text{ times}}\times\underbrace{(0,\infty)\times...\times(0,\infty)}_{m\text{ times}}$$
equipped with the product topology, consider its subspace $\mathcal{E}$ consisting of those points $(\underline{s}_1,...,\underline{s}_m,t_1,...,t_m)$ such that $\underline{s}_1,...,\underline{s}_m$ is the $m$-tuple of exponentially bounded non-overlapping external addresses (we allow (pre-)periodic addresses), and $t_i>t_{\underline{s}_i}$ for $i\in\{1,...,m\}$. Let also $\mathcal{E}^0$ be the subspace of $\mathcal{E}$ containing points $(\underline{s}_1,...,\underline{s}_m,t_1,...,t_m)$ for which all addresses $\underline{s}_1,...,\underline{s}_m$ are non-(pre-)periodic.

\begin{defn}[Wedge at $\alpha\in\mathcal{E}$]
	For a point $\alpha=(\underline{s}_1,...,\underline{s}_m,t_1,...,t_m)\in\mathcal{E}$, a \emph{wedge} at $\alpha$ is a set of the form $W_{t_1,K_1}\times...\times W_{t_m,K_m}\subset\mathcal{S}^m$ where $W_{t_1,K_1},...,W_{t_m,K_m}$ are wedges around $\underline{s}_1,...,\underline{s}_m$, respectively.
\end{defn}

Define the function
$$G:\mathcal{E}^0\to\mathcal{N}_{f_0}$$
so that for $G(\underline{s}_1,...,\underline{s}_m,t_1,...,t_m)$ is the unique function from $\mathcal{N}_{f_0}$ such that its singular values escape on rays as prescribed by $(\underline{s}_1,...,\underline{s}_m,t_1,...,t_m)$. We study the continuity properties of this function and prove that it extends to $\mathcal{E}$. 

In this paragraph we introduce some notation. Let $\alpha=(\underline{s}_1,...,\underline{s}_m,t_1,...,t_m)\in\mathcal{E}^0$. Then there is a well-defined Thurston's $\sigma_f$-map for a quasiregular map $f=c\circ f_0$ with a capture $c$. Further, consider a sequence $(\alpha_n)_{n=1}^\infty=(\underline{s}_1^n,...,\underline{s}_m^n,t_1^n,...,t_m^n)_{n=1}^\infty\subset\mathcal{E}^0$. Again, for each element of the sequence, there is the well-defined Thurston's $\sigma_{f_n}$-map for quasiregular maps $(f_n=c_n\circ f_0)_{n=1}^\infty$ with captures $(c_n)_{n=1}^\infty$. Note that generally speaking all $\sigma_{f_n}$ act on different \tei\ spaces. By $\varphi$ and $\varphi_n$ we denote $\id$-type maps for the Thurston's iteration associated to $\sigma_f$ and $\sigma_{f_n}$, respectively. Also, let $a_{ij}^n$ be the corresponding points on the singular orbit of $f_n$, while $\id_n$ be the identity map corresponding to $\sigma_{f_n}$ (this is needed to avoid confusion between different \tei\ spaces).

Generally speaking, one can choose the capture in many different ways (even up to homotopy) and the Thurston iteration depends on this choice (though provides the same entire function realizing the required singular dynamics). In the next lemma we show that in some cases the choice of the captures can be made in a consistent way, so that they yield ``similar'' Thurston's iterations.   

\begin{lmm}[Similar pull-back]
	\label{lmm:same_pull_back}
	Consider $\alpha\in\mathcal{E}^0$, and let $(\alpha_n)_{n=1}^\infty\subset\mathcal{E}^0$ be a sequence converging to $\alpha$ so that $(\underline{s}_1^n,...,\underline{s}_m^n)_{n=1}^\infty$ is contained in a wedge at $\alpha$. For all big enough $n$, the captures $c$ and $(c_n)_{n=1}^\infty$ can be chosen in such a way that whenever the subspiders $S_\varphi^0$ and $S_{\varphi_n}^0$ (corresponding to different \tei\ spaces!) are equivalent, we have $\varphi\circ f\circ\hat{\varphi}^{-1}=\varphi_n\circ f\circ\hat{\varphi}_n^{-1}$. 
\end{lmm}
\begin{proof}
	A particular choice of $c$ is irrelevant so just pick an arbitrary one. Let $v_1,...,v_m$ be the finite singular values of $f_0$. Surround each $c(v_i)$ by an open disk $B_i$, so that these disks are mutually disjoint and each $B_i$ intersects only one leg of the standard spider of $f$ (the one having $c(v_i)$ as its endpoint).
	
	Due to Lemma~\ref{lmm:continuity_rays_external_address}, without loss of generality we might assume that $a_{i0}^n\in B_i$ for all $n$. Define $c_n$ to be the capture $c$ postcomposed with a \qc\ homeomorphism $\theta_n:\mathbb{C}\to\mathbb{C}$ mapping each $a_{i0}$ to $a_{i0}^n$ and equal to identity on the complement of $\cup_i B_i$. Note that due to Theorem~\ref{lmm:continuity_rays_external_address} legs $R_{ij}^n$ converge to $R_{ij}$ as $n\to\infty$, and $R_{ij}^n$ converge to $\infty$ (as sets) as $j\to \infty$ uniformly in $n$. This means that for $R_{i0}^n$ (corresponding to $f_n$) is the only leg intersecting $B_i$ for all big enough $n$. Hence, such $c_n$ indeed satisfies our definition of capture. Assume that this holds for all $n$.  
	
	Let $\id$ and $\id_n$ be the identity maps corresponding Thurston's iterations via $f$ and $f_n$, respectively. Next, again from Theorem~\ref{lmm:continuity_rays_external_address}, after increasing $n$ we might assume that the subspiders  $\theta_n (S_{\id}^0)$ and $S_{\id_n}^0$ are equivalent: in fact we get $S_{\id_n}^0$ (up to leg homotopy) from $S_{\id}^0$ by moving each $a_{i0}$ to $a_{i0}^n$ inside of $B_i$. 
	
	Since $\varphi_n\circ c_n=(\varphi_n\circ\theta_n)\circ c$, we have $[\varphi_n\circ\theta_n]=[\varphi]$ in the \tei\ space of the complement of the singular values of $f$, i.e.\ of $\mathbb{C}\setminus\{a_{10},...,a_{m0}\}$. Further, $[\varphi_n\circ c_n]=[\varphi\circ c]$ in the \tei\ space of the complement of the singular values of $f_0$. But then $\varphi\circ f\circ\hat{\varphi}^{-1}=\varphi_n\circ f\circ\hat{\varphi}_n^{-1}$. 
\end{proof}

Before proceeding, we need to prove a technical lemma.

\begin{lmm}[Expansivity near $\infty$]
	\label{lmm:expansivity}
	Fix a disk $D_\rho=\{z: \abs{z}<\rho\}$. For all $w\in\mathbb{C}$ big enough and every continuous path $g_u,u\in[0,1]$ of functions (coefficients' continuity) from $\nd$ such that $\SV(g_u)\subset D_\rho$, holds
	$$z_1=z_0+O(1/\abs{w})$$
	where $z_u,u\in[0,1]$ is a continuous path in $\mathbb{C}$ satisfying $g_u(z_u)=w$.
\end{lmm}
\begin{proof}
	This lemma is basically a simplified version of \cite[Proposition 5.3 (3)]{IDTT2}. Since the singular values of $g_u$ belong to the bounded set $D_\rho$, the coefficients of $g_u$ are bounded by some constant depending only on $\rho$. All the computations below make sense for big enough $\abs{w}$. Now, let $g_u(z_u)=w$. So $g_u(z_u)=g_0(z_0)$, and
	$$e^{dz_u}(1+O(e^{-z_u}))=e^{dz_0}(1+O(e^{-z_0})).$$
	After taking logarithm of both sides, we obtain
	$$z_u=z_0+\frac{2\pi in_u}{d}+O(e^{-z_u})+O(e^{-z_0}).$$
	Due to continuity of $z_u$, we must have $n_u=0$. Further, since $g_u(z_u)=w$, we have $e^{dz_u}(1+O(e^{-z_u}))=w,$ and $e^{\Re dz_u} (1+O(e^{-z_u}))=\abs{w}.$ Hence, we may assume that
	$$\Re z_u>C\log\abs{w}$$
	for some constant $C>0$. But then $z_1=z_0+O(1/\abs{w})$.	
\end{proof}

The next proposition is the most important in our constructions. Roughly speaking, it tells that the nearby points in $\mathcal{E}^0$ allow uniform choice of $\rho$ in the construction of the compact invariant subset (Theorem~\ref{thm:invariant_subset}). Note, that the integers $N_i$ correspond to $\mathcal{C}_f(\rho)$ (they are defined in Subsection~\ref{subsec:clusters}).

\begin{prp}[Invariant subspiders]
	\label{prp:invariant_subspiders}
	Consider $\alpha\in\mathcal{E}^0$, and let $(\alpha_n)_{n=1}^\infty\subset\mathcal{E}^0$ be a sequence converging to $\alpha$ so that $(\underline{s}_1^n,...,\underline{s}_m^n)_{n=1}^\infty$ is contained in a wedge at $\alpha$, and the captures $c$ and $(c_n)_{n=1}^\infty$ are chosen so that the conclusions of Lemma~\ref{lmm:same_pull_back} hold.
	
	There is a $K_0>\max N_i$ and a sequence of positive integers $(n_K)_{K=K_0}^\infty$ such that for every pair $(K,n)$ with $K\geq K_0$ and $n>n_K$, there is a $\sigma_{f_n}$-invariant subset $\mathcal{B}_{n,K}$ of $\id$-type points of the \tei\ space $\mathcal{T}_{f_n}$ containing $[\id_n]$, for which holds the following condition: every point $\mathcal{B}_{n,K}$ is represented by an $\id$-type map $\varphi_n$ such that $S_{\varphi_n}^K$ is equivalent to $S_\varphi^K$ for some $\id$-type map $\varphi$ satisfying conditions (1)-(6) of Theorem~\ref{thm:invariant_subset}.
\end{prp}
\begin{remark}
	Note that the subspiders $S_{\varphi_n}^K$ and $S_\varphi^K$ are subspiders of $S_{\varphi_n}$ and $S_\varphi$ which correspond to different \tei\ spaces.
\end{remark}
\begin{proof}
	Let $\mathcal{B}_{n,K}$ be a subset of $\mathcal{T}_{f_n}$ represented by $\id$-type maps $\varphi_n$ for which there is an isotopy $\psi_u,u\in[0,1]$ of $\id$-type maps between $\psi_0=\id_n$ and $\psi_1=\varphi_n$ satisfying the following conditions:
	\begin{enumerate}
		\item for every  $u\in[0,1]$, $S_{\psi_u}^K$ is equivalent to $S_{\varphi_u}^K$ for an isotopy $\varphi_u, u\in[0,1]$ of $\id$-type maps satisfying conditions (1)-(6) of Theorem~\ref{thm:invariant_subset};
		\item if $j>K$, then
		$$\abs{\psi_u(a_{ij}^n)-a_{ij}^n}<\frac{1}{2j};$$	
		\item for every pair $a_{ij},a_{kl}$ with $\max\{i,j\}>K$ belonging to the same cluster and satisfying $\max\{i,j\}-K<L(a_{ij},a_{kl})$, holds
		$$\psi_u(a_{kl}^n)-\psi_u(a_{ij}^n)=\frac{2\pi i(s_{k(l+H)}-s_{i(j+H)})}{d(F^{\circ H}(t))'}\nu_u\delta_u$$
		where $t$ is the potential of $a_{ij}$ and $a_{kl}$, $H=H(a_{ij},a_{kl})$, $\nu_u=\nu_u(\psi_u,i,j,k,l)\in A_{t,H-1}$ and $\delta_u=\delta_u(\psi_u,i,j,k,l)\in D_{i(j+H)}^{k(l+H)}$.
	\end{enumerate}
	
	We want to find such $K_0$ that for $K>K_0$ and all $n$ big enough, the set $\mathcal{B}_{n,K}$ is invariant under $\sigma_{f_n}$ and contains $[\id_n]$. Due to Lemma~\ref{lmm:continuity_rays_external_address}, $[\id_n]$ is obviously contained in $\mathcal{B}_{n,K}$ for every $K>\max_i\{N_i\}$ and all big enough $n$.
	
	Let $\chi_u, u\in[0,1]$ be the concatenation of $\psi_u$ with the capture isotopy (see the construction of the capture in Subsection~\ref{subsec:setup_and_contraction}) between $\id_n$ and $c_n^{-1}$, i.e.\ $\chi_0=c_n^{-1}$ and $\chi_1=\varphi_n$. Since the capture isotopy ``moves'' only the singular values, $\chi_u$ will satisfy the conditions (1)-(3) for $K>\max_i\{N_i\}$ and all big enough $n$.
	
	Now, we need to check when conditions (1)-(3) hold for $\hat{\chi}_u$. Existence of $K_0$ such that condition (2) holds for every $K\geq K_0$ and big enough $n$ follows from Lemma~\ref{lmm:expansivity} and Lemma~\ref{lmm:continuity_rays_external_address} (2), from which we get lower bounds on the absolute value of $a_{ij}^n$.
	
	For a fixed $K$, after making $n$ bigger we might assume that for all $a_{ij}$ involved in condition (3), holds $\abs{a_{ij}^n-a_{ij}}<1/2j$. Hence, from the triangle inequality follows that $\abs{\chi_u(a_{ij}^n)-a_{ij}}<1/j$. Let $a_{i(j+1)},a_{k(l+1)}$ be a pair from condition (3). Since both $a_{i(j+1)}$ and $a_{k(l+1)}$ are outside of $D_\rho$, we can construct an isotopy $\xi_u,u\in[0,1]$ of $\id$-type maps satisfying conditions of Proposition~\ref{prp:good_big_disk_around_origin_cluster_estimates} such that for $z=a_{pq}\in\{a_{10},a_{20},...,a_{m0},a_{i(j+1)},a_{k(l+1)}\}$ and $u\in[1/2,1]$, we have $\xi_u(a_{pq})=\psi_{2u-1}(a_{pq}^n)$, that is, the marked points $\xi_u(a_{pq})$ have the same trajectories as $\psi_{2u-1}(a_{pq}^n)$ for $u\in[1/2,1]$. Then from Proposition~\ref{prp:good_big_disk_around_origin_cluster_estimates} follows that condition (3) holds for $\hat{\psi}_u$ and $a_{ij},a_{kl}$.
	
	Further, since conditions (2) and (3) hold in particularly for $a_{ij}$ with $j=K+1$, for $n$ big enough, $S_{\varphi_n}^{K+1}=S_{\psi_1}^{K+1}$ is equivalent to $S_{\varphi}^{K+1}$ for some $\id$-type map $\varphi$ satisfying conditions (1)-(6) of Theorem~\ref{thm:invariant_subset}. But then, since due to Lemma~\ref{lmm:same_pull_back} equivalent subspiders $S_{\varphi_n}^0$ and $S_\varphi^0$ induce the same pullback (i.e.\ $\varphi_n\circ f\circ\hat{\varphi}_n^{-1}=\varphi\circ f\circ\hat{\varphi}^{-1}$), we obtain that $S_{\hat{\varphi}_n}^K$ is equivalent to $S_{\hat{\varphi}}^K$. 
\end{proof}

We are ready to prove a ``conditional'' continuity of $G$ on on $\mathcal{E}^0$.

\begin{thm}[Continuity of $G$ on $\mathcal{E}^0$]
	\label{thm:continuity_non_periodic}
	Consider $\alpha\in\mathcal{E}^0$, and let $(\alpha_n)_{n=1}^\infty\subset\mathcal{E}^0$ be a sequence converging to $\alpha$ so that $(\underline{s}_1^n,...,\underline{s}_m^n)_{n=1}^\infty$ is contained in a wedge at $\alpha$. Then	$G(\alpha_n)\to G(\alpha)$.		
\end{thm}
\begin{proof}
	Consider the sequence $(\varphi'_n)_{n}^\infty$ of $\id$-type maps such that $[\varphi'_n]$ is the fixed point of $\sigma_{f_n}$ belonging to $\mathcal{C}_{f_n}$. From to Theorem~\ref{prp:invariant_subspiders} we see that for every index $i,j$ the sequence $(\varphi'_n(a_{ij}^n))_{n}^\infty$ is bounded. Indeed, if we choose $K>j$, then for all $n$ big enough, $\varphi'_n(a_{ij}^n)$ will have the same bounds as $\varphi(a_{ij})$ in Theorem~\ref{thm:invariant_subset}, because $\sigma_{f_n}[\id_n]$ converges to $[\varphi'_n]$.
	
	After changing to a subsequence we may assume that every sequence $\varphi'_n(a_{ij}^n)$ converges in $\mathbb{C}$. Since the homotopy of $S_{\varphi'_n}^0$ is bounded, the sequence $(G(\alpha_n))_{n=1}^\infty$ has at least one limiting point $\mathcal{N}_{f_0}$. Pick one of them and denote it by $g$. 
	
	Let $\tilde{R}_{ij}^n$ be the ray tail of $\varphi'_n$ which ends at $\varphi'_n(a_{ij}^n)$. Then due to Lemma~\ref{lmm:continuity_rays_function}, there exists a limit $\tilde{R}_{ij}:=\lim_{n\to\infty}\tilde{R}_{i0}^n$ which is a ray tail of $g$ for all $j$ big enough. Since the external addresses are non-overlapping, the same is true for all $j\geq 0$ (usual pull-back argument). Clearly, $\tilde{R}_{i0}$ has external address $\underline{s}_{i}$, so $g$ is the function from $\mathcal{N}_{f_0}$ realizing the singular escaping dynamics encoded by $\alpha$. From \cite[Theorem 1.1]{IDTT3} we know that such function is unique, i.e.\ $g$ is the only limiting point. Hence, $G(\alpha)=g$.
	
	So we proved that every sequence $(\alpha_n)_{n=1}^\infty$ converging to $\alpha$ as in the statement of the theorem has a subsequence $(\alpha_{n_k})_{k=1}^\infty$ such that $G(\alpha_{n_k})$ converges to $G(\alpha)$. Claim follows.
\end{proof}

Note, that we have not really used the existence of the entire function realizing the combinatorics encoded by $\alpha$, but rather used the invariant set of conditions from Theorem~\ref{thm:invariant_subset}. Using the same constructions we can continuously extend $G$ to $\mathcal{E}$.

\begin{proof}[Proof of Theorem~\ref{thm:main_thm}]
	The conditions (1)-(6) of Theorem~\ref{thm:invariant_subset} define an invariant subset of $\id$-type functions also in the case when some of the external addresses $\underline{s}_i$ are (pre-)periodic (see also \cite[Remark~4.3]{IDTT3}).
	
	For any $\alpha\in\mathcal{E}$, we can find a sequence $(\alpha_n)_{n=1}^\infty\subset\mathcal{E}^0$ converging to $\alpha$ so that $(\underline{s}_1^n,...,\underline{s}_m^n)_{n=1}^\infty$ is contained in a wedge at $\alpha$. Lemma~\ref{lmm:same_pull_back} and Proposition~\ref{prp:invariant_subspiders} stay valid also in this context, so repeating the argument of Theorem~\ref{thm:continuity_non_periodic} we obtain the function $g$ realizing the singular dynamics encoded by $\alpha$. Uniqueness of such function can be proved exactly as in \cite[Theorem~1.1]{IDTT1}. So the function $G$ is well-defined on $\mathcal{E}$.
\end{proof}

Finally, we state and prove the general continuity result.

\begin{thm}[Continuity of $G$ on $\mathcal{E}$]
	\label{thm:continuity}
	Consider $\alpha\in\mathcal{E}$, and let $(\alpha_n)_{n=1}^\infty\subset\mathcal{E}$ be a sequence converging to $\alpha$ so that $(\underline{s}_1^n,...,\underline{s}_m^n)_{n=1}^\infty$ is contained in a wedge at $\alpha$. Then $G(\alpha_n)\to G(\alpha)$.		
\end{thm}
\begin{proof}
	Note that every $\alpha_n\in\mathcal{E}\setminus\mathcal{E}^0$ can be approximated by points from $\mathcal{E}^0$ contained in the same wedge at $\alpha$. Hence, if there is a sequence $(\alpha_n)_{n=1}^\infty\subset\mathcal{E}$ converging to $\alpha$ and such that $G(\alpha_n)$ does not converge to $G(\alpha)$, one can find a sequence $\alpha'_n\in\mathcal{E}^0$ converging to $\alpha$ and such that $G(\alpha'_n)$ does not converge to $G(\alpha)$ which is impossible, as can be seen from the proof of Theorem~\ref{thm:main_thm}.
\end{proof}

Theorem~\ref{thm:parameter rays} is special case of the previous theorem.

\section{Acknowledgements}

We would like to express our gratitude to our research team in Aix-Marseille Universit\'e, especially to Dierk Schleicher who supported this project from the very beginning, Sergey Shemyakov who carefully proofread all drafts, as well as to  Kostiantyn Drach, Mikhail Hlushchanka, Bernhard Reinke and Roman Chernov for uncountably many enjoyable and enlightening discussions of this project at different stages. We also want to thank Dzmitry Dudko for his multiple suggestions that helped to advance the project, Lasse Rempe for his long list of comments and relevant questions, and Adam Epstein for important discussions especially in the early stages of this project. 

Finally, we are grateful to funding by the Deutsche Forschungsgemeinschaft DFG, and the ERC with the Advanced Grant “Hologram” (695621), whose support provided excellent conditions for the development of this research project. This research received support from the European Research Council (ERC) under the European Union Horizon 2020 research and innovation programme, grant 647133 (ICHAOS).

\vspace{1em}
Aix-Marseille Universit\'e, France

\textit{Email:} bconstantine20@gmail.com


\begin{thebibliography}{RSSS}
	
	
	\bibitem[B0]{MyThesis} Konstantin Bogdanov, \emph{Infinite-dimensional Thurston theory and transcendental dynamics with escaping singular orbits}. PhD Thesis, Aix-Marseille Universit\'e, 2020.
	
	\bibitem[B1]{IDTT1} Konstantin Bogdanov, \emph{Infinite-dimensional Thurston theory and transcendental dynamics I: infinite-legged spiders}. arXiv:2102.00300
	
	\bibitem[B2]{IDTT2} Konstantin Bogdanov, \emph{Infinite-dimensional Thurston theory and transcendental dynamics II: classification of entire functions with escaping singular orbits}. arXiv:2102.10728
	
	\bibitem[B3]{IDTT3} Konstantin Bogdanov, \emph{Infinite-dimensional Thurston theory and transcendental dynamics III: entire functions with escaping singular orbits in the degenerate case}. arXiv:2104.12206
	
	
	
%
%
%
%
%
%
%
	
%
%
%
	
%
%
%
	
	\bibitem[F]{MarkusThesis}
	Markus F\"orster, 
	\emph{Exponential maps with escaping singular orbits}. 
	PhD Thesis, International University Bremen, 2006.
%
	
	\bibitem[FRS]{FRS} 
	Markus F\"orster, Lasse Rempe, and Dierk Schleicher,
	\emph{Classification of escaping exponential maps}. Proc Amer Mathe Soc. \textbf{136} (2008), 651--663.
	
	\bibitem[FS]{MarkusParaRays} 
	Markus F\"orster and Dierk Schleicher,  \emph{Parameter rays in the space of exponential maps}. Ergod Theory Dynam Systems \textbf{29} (2009), 515--544.
	
	
%
%
%
%
%
%
	
%
%
%
	
	\bibitem[RRRS]{RRRS}
	G\"unter Rottenfu{\ss}er, Johannes R\"uckert, Lasse Rempe, and Dierk Schlei\-cher, \emph{Dynamic rays of bounded-type entire functions}. {Annals of Mathematics} \textbf{173} (2011), 77--125.
	
%
%
	
%
%
	
	\bibitem[SZ]{SZ-Escaping}
	Johannes Zimmer and Dierk Schleicher, 
	\emph{Escaping points for exponential maps}. 
	J. Lond. Math. Soc. (2) \textbf{67} (2003), 380--400.
	
\end{thebibliography}
\end{document}